\documentclass{article}
\usepackage{arxiv}
\usepackage{float}
\usepackage[utf8]{inputenc} 
\usepackage[T1]{fontenc}    
\usepackage{hyperref}       
\usepackage{url}            
\usepackage{booktabs}       
\usepackage{amsfonts}       
\usepackage{nicefrac}       
\usepackage{microtype}      
\usepackage{lipsum}
\usepackage{enumitem}
\usepackage{graphicx}
\usepackage{amsmath}
\usepackage[pagewise]{lineno}
\graphicspath{ {./images/} }
\usepackage{amsthm}
\usepackage[skip=2pt]{caption}
\usepackage[figurename=Fig.]{caption}
 \usepackage{epsfig}
 \usepackage{algorithm}
 \usepackage{diagbox}
 \usepackage{adjustbox}
 \usepackage[noblocks]{authblk}
\usepackage{wasysym}
\usepackage[utf8]{inputenc}

\newtheorem{thm}{Theorem}[section]
\newtheorem{thm*}{Theorem}[section]
\newtheorem{lem}{Lemma}[section]
\newtheorem{lem*}{Lemma}[section]
\theoremstyle{definition}

\newtheorem{cor}{Corollary}[section]
\newtheorem{exmp}{Example}[section]

\title{Colorings of some Cayley graphs}
\author{Prajnanaswaroopa S\\
sntrm4@rediffmail.com}

\date{sntrm4@rediffmail.com}

\begin{document}

	\maketitle
\textbf{Abstract:} Cayley graphs are graphs on algebraic structures, typically groups or group-like structures. In this paper, we have obtained a few results on Cayley graphs on Cyclic groups, powers of cycles, Cayley graphs on some non-abelian groups, and vertex, edge and total colorings of Cayley graphs on gyrogroups.
 
\section{Introduction}
	For a simple loopless graph $G$, we denote by $V(G)$ and $E(G)$ the vertex and edge sets of the graph, respectively. A $k$-vertex coloring of a graph $G$ is a map $c:V(G)\to\{1,2,\ldots,k\}$ such that $c(v_i)\neq c(v_j)$, where $v_i, v_j\in V(G)$ are adjacent vertices. The minimum $k$ required to color vertices is called the chromatic number of $G$, denoted by $\chi(G)$.
	
	Edge coloring of a graph $G$ is the proper coloring of the edges of $G$ such that no two edges incident on the same vertex receive the same color. It can also be interpreted as the vertex coloring of its line graph, $L(G)$. In terms of mappings, a $k$-edge coloring of $G$ is a map $c:E(G)\to\{1,2,\ldots,k\}$ such that $c(e)\neq c(e')$ for any two incident edges $e,e' \in E(G)$. The minimum $k$ required in such a coloring is the edge chromatic number, or the chromatic index of $G$, denoted by $\chi'(G)$. By Vizing's theorem (\cite{VIZ}), it is known that $\chi'(G)$ is either $\Delta(G)$ or $\Delta(G)+1$, where $\Delta(G)$ is the maximum degree of the graph $G$. The graphs $G$ with $\chi'(G)=\Delta(G)$ are said to be of class I, and those with $\chi'(G)=\Delta(G)+1$ are said to be of class II.
	
	The total Coloring of a graph $G$ is the Coloring of the elements of $G$ such that no two adjacent vertices, two adjacent edges, or an edge and its incident vertices receive the same color. In other words, a $k$-total coloring is a map $c:V(G)\cup E(G)\to\{1,2,\ldots,k\}$ such that $c(u)\neq c(v)$ for any two adjacent vertices $ u, v\in V(G) $,$c(e)\neq c(e')$ for any two incident edges $e,e' \in E(G)$ and $c( v)\neq c(e)$ for any vertex $ v\in V(G)$ and any edge $ e\in E(G)$ incident to $v$. The minimum $k$ required in such a coloring is called the total chromatic number of the graph, denoted by $\chi''(G)$. A trivial bound on total Coloring is that $\chi''(G)\ge\Delta(G)+1$, where $\Delta(G)$ is the maximum degree of $G$. Total Coloring Conjecture(TCC) is the assertion that $\chi''(G)\le\Delta(G)+2$ (\cite{BEH}, \cite{VIZ1}). The graphs with $\chi''(G)=\Delta(G)+1$ are called type I, and those with $\chi''(G)=\Delta(G)+2$ are said to be type II. 
	
	In this paper, we obtain some bounds on the total chromatic number of some Cayley graphs on symmetric groups and powers of cycles, which are a class of Cayley graphs on Cyclic groups. We also obtain bounds on the chromatic number, chromatic index, and total chromatic number of some classes of Cayley graphs on gyrogroups. The Cayley graphs on a group/ gyrogroup $\Gamma$ with symmetric generating set $S$ (a set is called symmetric if both $s$ and $s^{-1}$ both belong to $S$) will be denoted by $C(\Gamma, S)$. As all graphs are loopless, $S$ does not have the identity element of the group/gyrogroup. The symmetric group, which is the group of all bijective functions from an $n$-element set to itself, will be denoted by $S_n$. $ C_n^k$ will denote the $k$-th power of $n$-cycle.
	
	\section{Some Results on Cayley graphs of non-abelian groups}
	Very few results on total coloring Cayley graphs on non-abelian groups have been known. Some related results can be found in (\cite{PRA2}). Here is a continuation of those results. We begin with a lemma.
 \begin{lem}
 If $x\in S_n$ and $1\le i<n$ is an integer, and  $x=(1,2,\ldots,n)^i(1,2)$, then $x=s(1,2,\ldots,n)^j$ for some $1\le j<n$ and $s\in S_{n-1}$.
 \end{lem}
 \begin{proof}
 We see that $(1,2,\ldots,n)^i$, $1\le i<n$ takes $n-i$ to $n$ and $n$ to $i$. Therefore, we can always choose $k$ in such a manner that $(1,2,\ldots,n)^i(1,2)(1,2,\ldots,n)^k$ keeps $n$ fixed. For $i=1$, we can take $k=n-2$, for $i=2$, we can take $k=n-1$. In all other cases,$i\leq 3 < n$, we can take $k=n-i$. Choosing $j=n-k$ gives us the desired conclusion with $s=(1,2,\ldots,n)^i(1,2)(1,2,\ldots,n)^{n-j}$. 
 \end{proof}
  \begin{thm}
	 	The graph $G=C(S_n,S)$ with\ $S=\{(1,2),(1,2,\ldots,n),(1,n,\ldots,2) \}$ with $3|n$ is type I.
	 \end{thm}
	 \begin{proof}
	We observe that the graph $G$ is a disjoint union of a perfect matching generated by the element $(1,2)$ and cycles generated by the elements $(1,2,\ldots,n)$ and $(1,n,\ldots,2)$. In order to perform the total Coloring, we first color the matching formed by the element $(1,2)$ by a unique color. The remaining graph is  a union of disjoint cycles. Now, in order to color the graph totally, we take left cosets of the cyclic group $H=\langle(1,2,\ldots,n)\rangle$, which has order $n$ with respect to $S_n$. The cosets are labelled as $g_iH$, where $g_i\in S_{n-1}$ and $i$ ranges from $1$ to $(n-1)!$.  Using Lemma 2.1, $G$ consists of disjoint $n$ cycles ($(n-1)!$ copies), which correspond to the vertices of subgroup $H$ and any two copies are joined by a unique edge with respect to the generating element $(1,2)$. By using the result on (\cite{CHEN}) equitable $\Delta(G)$ coloring for $\Delta(G)\le 3$, we can equitably three color the vertices of $G$. In addition, this structure shows that if the 1-factor due to the generating element $(1,2)$ is removed, we have a disjoint union of $n$ cycles which correspond to $(n-1)!$ copies of the subgraph induced by the subgroup $H$. The edges of the disjoint union of cycles can always be colored by the same three colors used to color the vertices because we know that if $3|n$, then every $n$ cycle is $3$ total colorable with arbitrary vertex coloring. Therefore, we can color the graph $G$ totally using $4$ colors; in other words, $G$ is type I.
	 \end{proof}
	\section{Some results on Powers of Cycles}
	It is proved that all graphs $C_n^k$ with $n=s(2m+1)-1\quad, \ \frac{k}{2}\le m\le k, \ s-\text{even}$ satisfies TCC \cite{PRA}. We show some related results in the below discussions. Some regard conformability, and others focus on the total Coloring of powers of cycles.  
	The following two theorems are also proved in \cite{ZOR}. We provide a shorter proof.
	\begin{thm}
		Every graph $G=C_n^k$ with $n$ even is conformable. 
	\end{thm}
	\begin{proof}
		If we have $k+1\ge\frac{n}{4}$, then we could divide the vertices as $[0,\frac{n}{2}],[1,\frac{n}{2}+1],\ldots,[\frac{n}{2}-1,n-1]$ to get a conformable coloring of the vertices. On the other hand, if $k+1<\frac{n}{2^m}\,,m\ge2$, we can divide the vertices into $x+y$ color classes, where $x=\lfloor\frac{n}{2^m}\rfloor$ and $y=\frac{n-\lfloor\frac{n}{2^m}\rfloor}{2}$. We have $2^m$ vertices in $x$ color classes and $2$ vertices in $y$ color classes. The Coloring is conformable, as all the independent sets so divided have even parity (including the null independent sets). 
	\end{proof}
	\begin{thm}
		If $k+1<\frac{n}{3}$, then the power of cycle graphs $C_n^k$ are conformable.
	\end{thm}
	\begin{proof}
		By the previous theorem, we need to only consider the case when $n$ is odd. The given condition implies that the independence number of the graph $C_n^k$ is greater than or equal to $3$. Therefore, we can put $3$ vertices in an independent set. This condition is essential because $3$ is the minimum odd number after $1$, and the complete graphs of odd order are the only regular and conformable graphs having exactly $1$ vertex in all the independent sets. As the parity of $n$ is odd, we must have no null independent sets, for $0$ has even parity. So, let us first divide the first $2k+1$ vertices into $2k+1$ independent sets. Now, the induced graph formed by the remaining $n-(2k+1)$ vertices can be thought of as an induced subgraph of $k$-th power of a cycle of even order (it is an even power of path $P_{n-(2k+1)}^k$). Since we have proved that the even powers of cycles are vertex conformable, this implies that the induced subgraph is also vertex conformable. Thus, the graph $G$, in this case, is also vertex conformable, as we can arrange the remaining $n-(2k+1)$ vertices like the conformable Coloring of the induced graph of $k$-th power of even cycle, which ensures odd vertices in all the independent sets of $G$. 
	\end{proof}
	The above fact of conformability is a strong evidence in favor of the Campos-de Mello conjecture \cite{CAMP}. Though there are conformable graphs that are not type I, the high symmetry of powers of cycles makes it appear that conformability implies type I for these graphs.
	
	\begin{thm} If $n=m(k+1)+1$ for some even integer $m$ and odd $k$, then $G=C_n^k$ satisfies TCC.
	\end{thm}
	\begin{proof}
		We show how we can structure the total color matrix as follows. 
		We can say that the total color matrix can be divided into seven parts, of which the first part consists of the super-diagonal starting from the edge $(0,(n-k))$ and ending at the super-diagonal or entry corresponding to edge $(0,(n-1))$. The second part is formed of several subparts that consist of the diagonal region along with $k$ super diagonals. The third part corresponds to entries in the last column from the entry corresponding to the edge $((n-k-1),(n-1))$. The fourth, fifth, and sixth parts are symmetric counterparts of the first, second, and third parts. For coloring the first and second parts (and hence their counterparts fourth and fifth parts), several interspersing Tableau are used, which are described in the following procedure would be followed. \\

		We know that there exists a commutative idempotent Latin square of order $k+2$ that we denote by  $C'$. We construct three  Tableaux $A, B$, and $C$ of orders $k, k$, and $k+1$, respectively. The Tableau $A$ is upper triangular whose entries in the main diagonal are $2k+2$, and the subsequent sub-diagonals are filled respectively with $2k, 2k-1$ down to $k+3$. Similarly, Tableau $B$ is lower triangular with main diagonal entries being $k+3$, and subsequent sub-diagonals increase values to $2k+2$. Tableau $C$ is the pseudo-Latin square of order $k+1$ derived from the commutative idempotent Latin square of order $k+2$.
		The Tabelau $A$, $B$, $C$ are shown respectively in Tables \ref{table:1}, \ref{table:2}.
		
		\begin{table}[!ht]
			\begin{adjustbox}{width=\textwidth}
				\begin{tabular}{|c|c|c|c|c|}\hline
					$k+3$&&&&\\
					\hline$k+4$&$k+3$&&&\\
					\hline$k+5$&$k+4$&$k+3$&&\\
					\hline\ldots&\ldots&\ldots&$k+3$&\\
					\hline$2k+2$&\ldots&\ldots&\ldots&$k+3$\\
					\hline
				\end{tabular}
				\quad
				\begin{tabular}{|c|c|c|c|c|}\hline
					$2k+2$&$2k+1$&$2k$&$\ldots$&$k+3$\\
					\hline &$2k+2$&$2k+1$&$\ldots$&$k+3$\\
					\hline &&$2k+2$&$2k+1$&$\ldots$\\
					\hline & & & $2k+2$ & $2k+1$\\
					\hline & & & & $2k+2$\\
					\hline
				\end{tabular}
		\end{adjustbox}
		\caption{Left-Tableau $B$ and Right-Tableau $A$.}
		\label{table:1}
	\end{table}
	
	\begin{table}[!ht]
		\begin{adjustbox}{width=\textwidth}
			\begin{tabular}{|c|c|c|c|c|c|c|}\hline
				1&$\frac{k+3}{2}+1$&$2$&$\ldots$&$\ldots$&$k+2$&$\frac{k+3}{2}$\\\hline
				$\frac{k+3}{2}+1$&$2$&$\frac{k+3}{2}+2$&$\ldots$&$\ldots$&$\frac{k+3}{2}$&1\\\hline
				$\ldots$&&&$\ldots$&&$\ldots$\\\hline
				$\frac{k+3}{2}$&1&$\frac{k+3}{2}+1$&$\ldots$&$\ldots$&$\frac{k+3}{2}-1$&$k+2$\\\hline
			\end{tabular}
			\quad
			\begin{tabular}{|c|c|c|c|c|c|}\hline
				1&$\frac{k+3}{2}+1$&2&$\ldots$&$\ldots$&$k+2$\\\hline
				$\frac{k+3}{2}+1$&2&$\frac{k+3}{2}+2$&$\ldots$&$\ldots$&$\frac{k+3}{2}$\\\hline
				$\ldots$&&&$\ldots$&$\ldots$\\\hline $k+2$&$\frac{k+3}{2}$&1&$\ldots$&$\ldots$&$k+1$\\\hline
			\end{tabular}
		\end{adjustbox}
		\caption{Left-Tableau $C'$ and its derivative Right-Tableau $C$.}
		\label{table:2}
	\end{table}

	The three Tableaux are arranged as given in Table \ref{table:3} to form the $n \times n$ total color matrix.
	
	We place $m$ copies of Tableau $C$ (of order $k + 1$) along the main diagonal as depicted in Table \ref{table:4}, where $n=m(k+1)+1$. Each cell represents a sub-matrix of order $k+1$. In the cell $(i,i+1)$, $0<i<m$, one copy of $B$ ($i$ odd) or $A^T$ ($i$ even) is placed bottom-left justified. Similarly, in
	the cell $(i+1,i)$ one copy of $B^T$ ($i$ odd) or $A$ ($i$ even) is placed top-right justified. One copy each of $A$ and $A^T$ are respectively placed at cells $(1,m)$ and $(m,1)$.
	
	In the arrangement shown, the rows of Tableau $B$ start from the second row and $(k+2)$-th column. This ensures that the Tableaux $B$ and  $A$ are differently placed so that none of the numbers in the color matrix coincide. Similarly, the last row of the last $B^T$ is at the $(k+2)$-th row. This ensures that the colors in the last $B^T$  and $A^T$ do not coincide.
	From the above observations, the entries in rows (columns) of Tableau $A$ placed at cell $(1,m)$ and Tableau $B$ at $(1,2)$ (at $(m-1, m))$, respectively do not coincide. By symmetry, the
	colors in $B^T$ at $(2,1)$ (at $(m,m-1)$) and $A^T$ at $(m,1)$ also do not coincide.
	
	\begin{table}[ht!]\centering
		\begin{tabular}{|c|c|c|c|c|c|}\hline
			$C$&\diagbox{$B$}{}&&&&\diagbox{}{$A$}\\\hline
			\diagbox{}{$B^T$\\}&$C$&\diagbox{$A^T$\\}{}&&&\\\hline
			&\diagbox{}{$A$}&$C$&\diagbox{$B$}{}&&\\\hline
			&&\diagbox{}{$B^T$\\}&C&\diagbox{$A^T$\\}{}&\\\hline
			&&&\diagbox{}{$A$}&$C$&\diagbox{$B$}{}\\\hline
			\diagbox{$A^T$\\}{}&&&&\diagbox{}{$B^T$\\}&$C'$\\\hline\end{tabular}
		\caption{Color matrix structure.}
		\label{table:3}
	\end{table}
	
	In addition to the above arrangement, the third part of the total color matrix (and hence the sixth part) is given the colors that are just continuation of the entries of the Tableau $C$ as seen in the commutative idempotent Latin square of order $k+2$ of which the Tableau $C$ is derived. In the above Table \ref{table:3}, the third part and sixth parts are included in Tableau $C'$. As the entries of $C'$ are entirely different from that of $A$, there will be no clashes in this case.
	
	The procedures above give us a total coloring of $G$ because there are no clashes between the entries in the matrix. Thus, the graph $G$ satisfies TCC. 
\end{proof}
\begin{exmp}
	Consider the power of cycle $C_{13}^5$. Here, we have $m=2$ and $k+1=6$. The total color matrix, in this case, is given in Table \ref{table:4}:
	\begin{table}[htbp]
		\begin{tabular}{|c|c|c|c|c|c|c|c|c|c|c|c|c|c|}
			\hline
			&0&1&2&3&4&5&6&7&8&9&10&11&12\\\hline
			0&1&5&2&6&3&7&&&12&11&10&9&8\\\hline
			1&5&2&6&3&7&4&8&&&12&11&10&9\\\hline
			2&2&6&3&7&4&1&9&8&&&12&11&10\\\hline
			3&6&3&7&4&1&5&10&9&8&&&12&11\\\hline
			4&3&7&4&1&5&2&11&10&9&8&&&12\\\hline
			5&7&4&1&5&2&6&12&11&10&9&8&&\\\hline
			6&&8&9&10&11&12&1&5&2&6&3&7&\\\hline
			7&&&8&9&10&11&5&2&6&3&7&4&1\\\hline
			8&12&&&8&9&10&2&6&3&7&4&1&5\\\hline
			9&11&12&&&8&9&6&3&7&4&1&5&2\\\hline
			10&10&11&12&&&8&3&7&4&1&5&2&6\\\hline
			11&9&10&11&12&&&7&4&1&5&2&6&3\\\hline
			12&8&9&10&11&12&&&1&5&2&6&3&7\\\hline
			
		\end{tabular}
		\caption{Total Color matrix of $C_{13}^5$}
		\label{table:4}
	\end{table}
\end{exmp}
\begin{exmp}
	Consider the graph $C_{25}^5$. Here, $m=4$. The total color matrix, in this case, is given in Table \ref{table:5}
	\begin{table}[htbp]
		\begin{adjustbox}{width=\textwidth}
			\begin{tabular}{|c|c|c|c|c|c|c|c|c|c|c|c|c|c|c|c|c|c|c|c|c|c|c|c|c|c|}
				\hline
				&0&1&2&3&4&5&6&7&8&9&10&11&12&13&14&15&16&17&18&19&20&21&22&23&24\\\hline
				0&1&5&2&6&3&7&&&&&&&&&&&&&&&12&11&10&9&8\\\hline
				1&5&2&6&3&7&4&8&&&&&&&&&&&&&&&12&11&10&9\\\hline
				2&2&6&3&7&4&1&9&8&&&&&&&&&&&&&&&12&11&10\\\hline
				3&6&3&7&4&1&5&10&9&8&&&&&&&&&&&&&&&12&11\\\hline
				4&3&7&4&1&5&2&11&10&9&8&&&&&&&&&&&&&&&12\\\hline
				5&7&4&1&5&2&6&12&11&10&9&8&&&&&&&&&&&&&&\\\hline
				6&&8&9&10&11&12&1&5&2&6&3&7&&&&&&&&&&&&&\\\hline
				7&&&8&9&10&11&5&2&6&3&7&4&12&&&&&&&&&&&&\\\hline
				8&&&&8&9&10&2&6&3&7&4&1&11&12&&&&&&&&&&&\\\hline
				9&&&&&8&9&6&3&7&4&1&5&10&11&12&&&&&&&&&&\\\hline
				10&&&&&&8&3&7&4&1&5&2&9&10&11&12&&&&&&&&&\\\hline
				11&&&&&&&7&4&1&5&2&6&8&9&10&11&12&&&&&&&&\\\hline
				12&&&&&&&&12&11&10&9&8&1&5&2&6&3&7&&&&&&&\\\hline
				13&&&&&&&&&12&11&10&9&5&2&6&3&7&4&8&&&&&&\\\hline
				14&&&&&&&&&&12&11&10&2&6&3&7&4&1&9&8&&&&&\\\hline
				15&&&&&&&&&&&12&11&6&3&7&4&1&5&10&9&8&&&&\\\hline
				16&&&&&&&&&&&&12&3&7&4&1&5&2&11&10&9&8&&&\\\hline
				17&&&&&&&&&&&&&7&4&1&5&2&6&12&11&10&9&8&&\\\hline
				18&&&&&&&&&&&&&&8&9&10&11&12&1&5&2&6&3&7&\\\hline
				19&&&&&&&&&&&&&&&8&9&10&11&5&2&6&3&7&4&1\\\hline
				20&12&&&&&&&&&&&&&&&8&9&10&2&6&3&7&4&1&5\\\hline
				21&11&12&&&&&&&&&&&&&&&8&9&6&3&7&4&1&5&2\\\hline
				22&10&11&12&&&&&&&&&&&&&&&8&3&7&4&1&5&2&6\\\hline
				23&9&10&11&12&&&&&&&&&&&&&&&7&4&1&5&2&6&3\\\hline
				24&8&9&10&11&12&&&&&&&&&&&&&&&1&5&2&6&3&7\\\hline
			\end{tabular}
		\end{adjustbox}
		\caption{Total Color matrix of $C_{25}^5$}
		\label{table:5}
	\end{table}
	
\end{exmp}

\section{Results on Cayley graphs on Gyrogroups}
First, we will begin with a lemma that helps determine the isomorphism of Cayley graphs defined on the same group. An exponent distribution of a subset $\Sigma$ of a generating set $S$ of a group is the set of exponents of the elements of $S$  with respect to the product of elements of $\Sigma$. That is, if $\Sigma=\{\sigma_1,\sigma_2,\ldots,\sigma_n\}$ and $S=\{s_1,s_2,\ldots,s_n\}$; then the exponent distribution of $\Sigma$ is the set $\{i_{11},i_{12},\ldots,i_{1n},i_{21},i_{22},\ldots,i_{2n},\ldots,i_{n1},i_{n2},\ldots,i_{nn}\}$, where $s_1=(\sigma_1^{i_{11}}\sigma_2^{i_{12}}\ldots\sigma_n^{i_{1n}}), s_2=(\sigma_1^{i_{21}}\sigma_2^{i_{22}}\ldots\sigma_n^{i_{2n}}),\ldots,s_n=(\sigma_1^{i_{n1}}\sigma_2^{i_{n2}}\ldots\sigma_n^{i_{nn}})$. The following lemma gives an algorithm to determine when two Cayley Graphs on the same group with different generating sets are isomorphic. 
\begin{lem}
Two Cayley graphs $G_1=C(\Gamma, S_1)$ and $G_2=C(\Gamma, S_2)$ are isomorphic if there are two generating subsets $\Sigma_1\subset S_1$ and $\Sigma_2\subset S_2$ of $G$ of same cardinality and their exponent distributions with respect to some permutation of the generating set are the same.
\end{lem}
\begin{proof}
The proof is immediate on noting that one could construct identical graphs by starting from the identity element of $\Gamma$ and the identical products of the elements of generating subsets $\Sigma_1$ or $\Sigma_2$. 
\end{proof}
The above algorithm could be in polynomial time if the base group $\Gamma$ is cyclic, as the following corollary shows.
\begin{cor}
Let $\Gamma$ be a cyclic group of order $n$. If there exist generating elements of $s_1\in S_1$ and $s_2\in S_2$ such that their exponent distributions are same then the graphs $G_1=C(\Gamma,S_1)$ and $G_2=C(\Gamma,S_2)$ are isomorphic. 
\end{cor}
\begin{proof}
The proof is immediate from the previous lemma once it is known that the cyclic groups have a single minimal generator, and the corollary assumes that $s_1,s_2$ are those generators.
\end{proof}
\begin{thm} Let $\Gamma$ be the $2$-gyrogroup of order $n=2^l$ presented in \cite{MAU}. Then, the graph $C(\Gamma, S)$ with $S=\{1,2,\ldots,k,m-k,m-k+1,\ldots,m-1,\frac{m}{2}+m \}$. Then $G$ satisfies TCC. Further, if the power of cycle $C_m^k$ is type I, then $G$ is also type I.
\end{thm}
\begin{proof}
For ease of reference, we state the binary operation of the $2$-gyrogroup presented in \cite{MAH} and Example $1$ of \cite{MAU}. The binary operation $\oplus$ is defined as:\\
$$i\oplus j=\begin{cases}i+j\pmod m\qquad\qquad\qquad\qquad\qquad\qquad, i\in T_1; j\in T_1\\\{i+j\pmod m\}+m\quad\qquad\qquad\qquad\qquad, i\in T_1 ;j\in T_2\\\{i+(\frac{m}{2}-1)j\}\pmod m)+m\quad\qquad\qquad, i\in T_2; j\in T_1\\\{(\frac{m}{2}-1)i+(\frac{m}{2}+1)j\pmod m\}+m\qquad, i\in T_2 ;j\in T_2 . \end{cases}$$
Here $T_1=\{0,1,\ldots,m-1\}$, $m=\frac{n}{2}$ and $T_2=\{m,m+1\ldots,n-1\}$.

We note that the subgraph induced by the set $T_1$ is the power of cycle $C_m^k$, because, we have $i\oplus j=i+j$ for $i,j\in\{0,1,\ldots,m-1\}$.

Similarly, the subgraph induced by the set $T_2$ with respect to the generating set $\{(\frac{m}{2}-1)+m,2m-2,\ldots,k(\frac{m}{2}-1)+m,2m-k(\frac{m}{2}-1),\ldots,2+m,(\frac{m}{2}+1)+m\}$ forms the power of cycle $C_m^k$. 

This is because the exponent distribution of $1$ for $\{1,2,\ldots,k,n-k,\ldots,n-2,n-1\}$ and $z=(\frac{m}{2}-1)$ for $\{(\frac{m}{2}-1),m-2,\ldots,k(\frac{m}{2}-1)m-k(\frac{m}{2}-1),\ldots,2,(\frac{m}{2}+1)\}$ are the same, in the group $\mathbb{Z}_m$. Therefore, we can use Corollary 5.2.2  to conclude the isomorphism.

Now, note that the element $\frac{m}{2}+m$ acts as a sort of reflection equivalent for the gyrogroup $\Gamma$, in the sense that, we have $(\frac{m}{2}+m)^{\oplus2}=(\frac{m}{2}+m)\oplus(\frac{m}{2}+m)=(\frac{m}{2}+1)(\frac{m}{2}+m)+(\frac{m}{2}-1)(\frac{m}{2}+m)\pmod m$ $\equiv2\frac{m^2}{4}=m(\frac{m}{2})=0\pmod m$. Therefore, this element induces a perfect matching in the graph $C(\Gamma, S)$ with the end vertices of the edges being in the two sets $T_1$ and $T_2$, respectively; because we have $i\oplus j=i+(\frac{m}{2}+1)j\pmod m$ for $i\in T_1$ $j\in T_2$.

Now, if $i$ and $j$ and in the same independent set of the subgraph induced by $T_1$, then $i\oplus (\frac{m}{2}+m)=(i+(\frac{m}{2}+m)\pmod m)+m$ and $j\oplus (\frac{m}{2}+m)=(i+(\frac{m}{2}+m)\pmod m)+m$ are also in the same independent set, as $i-j=i\ominus j$. Hence, we can divide the vertices of the induced subgraph on $T_2$ into the same number of independent sets as that of the induced subgraph on $T_1$ by shifting the translates of an independent set with respect to the reflection element $\frac{m}{2}+m$. Thus, the chromatic number of $G$ is the same as that of $C_m^k$. As the graphs $C_m^k$ satisfy TCC (by Theorem 16 of Campos-de Mello \cite{CAMP}), the induced subgraphs on $T_1$ and $T_2$ individually satisfy TCC. Since we have arranged the induced graphs together in the same number of independent sets, only the edge coloring of the connecting perfect matching between the induced subgraphs formed by the element $\frac{m}{2}+m$ needs to be done in order to complete the total Coloring of $G$. We give one extra color for this, making $G$ satisfy TCC. This argument also shows why $G$ will be of type I if $C_m^k$ is also of type I.
\end{proof}
The above theorem at once gives the following generalization as its corollary.
\begin{cor}
Let $\Gamma$ be the $2$-gyrogroup presented in \cite{MAU}. Then, if the graph $C(\mathbb{Z}_m, S_1)$ staisfies TCC, then $G=C(\Gamma,S)$ with $S_1\cup\{\frac{m}{2}+m\}$ satisfies TCC. Further, if the graph $C(\mathbb{Z}_m, S_1)$ is type I, then $G$ is also type I.
\end{cor}
\begin{proof}
The proof is immediate once we replace the element $1$ in the previous theorem with any suitable generator $s\in S_1$ of the group $\mathbb{Z}_m$. Then, the subgraphs induced by the two sets $T_1$ and $T_2$ are isomorphic circulant graphs. The element $\frac{m}{2}+m$ then gives us a perfect matching of elements in $G$ with the end vertices in $T_1$ and $T_2$, respectively. Again, by following the similar argument in the proof of the last theorem, the chromatic number of $G$ is the same as that of the graph $C(\mathbb{Z}_m, S_1)$. Then, giving the perfect matching induced by the element $\frac{m}{2}+m$ makes $G$ satisfy TCC. In a similar vein, if $C(\mathbb{Z}_m, S_1)$ is type I, $G$ is also type I.
\end{proof}
The following result, an immediate consequence of the previous theorem, gives us the chromatic index of such graphs.
\begin{thm}
Let $\Gamma$ be the $2$-gyrogroup presented in \cite{MAU}. Then, if the graph $C(\mathbb{Z}_m, S_1)$ satisfies TCC, then $G=C(\Gamma,S)$ with $S_1\cup\{\frac{m}{2}+m\}$ has chromatic number equal to that of the graph $\chi(C(\mathbb{Z}_m,S_1)).$ In addition, the graph is of class I.
\end{thm}
\begin{proof}
The chromatic number of $G$ can be inferred from the discussion of Corollary 5.2.4. As the circulant graphs $C(\mathbb{Z}_m, S_1)$ are of class I by the Corollary 2.3.1 of \cite{STO}, we need to only give one extra color to the perfect matching induced by the element $\frac{m}{2}+m$. Thus, all the edges of $G$ can be colored in precisely $\Delta(G)+1$ colors, or $G$ is of class I. 
\end{proof}
An immediate generalization of the above result in the context of edge coloring is the following:
\begin{thm}
All Cayley graphs $G=Cay(\Gamma, S)$ for the $2$-gyrogroup described in the theorems above and any generating set $S$ is of class I. 
\end{thm}
\begin{proof}
The proof is immediate once we note that all generating elements of the form $s=j+m$, where $j\in\{0,1,\ldots,m-1\}$ are sort of reflections (gyro-reflections), that is, satisfy the property $s^{\oplus2}=s\oplus s=(\frac{m}{2}-1)s+(\frac{m}{2}+1)s\pmod m\equiv ms\pmod s \equiv 0$. Hence, all the reflections give rise to perfect matchings in the graph $G$. Since the induced graph on the sets $\{0,1,\ldots,m-1\}$ and $\{m,m+1,\ldots,n-1\}$ are each circulant, therefore, by \cite{STO} Corollary 2.3.1 and the fact that the perfect matchings generated by any elements of the form $s$ can be $1$-factorized, the conclusion is immediate.
\end{proof}

\end{document}